\documentclass[11pt,t,colorinlistoftodos]{article}

\usepackage[left=2.5cm,right=2cm,top=2cm,bottom=2cm]{geometry}
\usepackage{%
	amsmath,%
	amsthm,%
	amssymb,%
	latexsym,%
	mathtools%
}
\usepackage[hyphens]{url}
\usepackage{hyperref}
\hypersetup{%
	naturalnames=false,%
	breaklinks=true,%
	colorlinks=true,
	frenchlinks=false,%
	linkcolor=black,%
	citecolor=blue,%
	debug=true,%
	filecolor=magenta,%
	urlcolor=blue,%
	linktoc=all,
}
\usepackage{float}
\usepackage{cleveref}
\usepackage[ruled]{algorithm2e}
\usepackage{tikz-cd} 				
\usepackage{graphicx} 
\usepackage{float}
\usepackage{tikz}
\usepackage{pgfplots}
\pgfplotsset{compat=1.15}
\usepackage{mathrsfs}
\usetikzlibrary{arrows}
\usepackage{tabularx, pgffor} 		
\usepackage{subfiles}				
\usepackage{caption}
\usepackage{subcaption}
\usepackage[symbol]{footmisc}
\newtheorem{theorem}{Theorem}[section]

\theoremstyle{definition}
\newtheorem{remark}{Remark}[section]
\newtheorem{example}{Example}

\definecolor{qqttzz}{rgb}{0,0.2,0.6}
\definecolor{qqwwtt}{rgb}{0,0.4,0.2}
\definecolor{ffzztt}{rgb}{1,0.6,0.2}
\definecolor{ccqqqq}{rgb}{0.8,0,0}

\date{}

\usepackage[%
	backend=bibtex,%
	style=numeric-comp,%
	hyperref,%
]{biblatex}

\makeatletter
\def\blfootnote{\xdef\@thefnmark{}\@footnotetext}
\makeatother

\begin{document}
	\title{On the number of small Steiner triple systems with Veblen points}
	\maketitle
	\blfootnote{\textit{\textup{2020} Mathematics Subject Classification}: 05B07, 20N05, 51E10.} 
	\noindent
	{{Giuseppe Filippone}
		\\
		\footnotesize{Dipartimento di Matematica e Informatica}\\
		\footnotesize{Università degli Studi di Palermo, Via Archirafi 34, 90123 Palermo, Italy}\\
		\footnotesize{giuseppe.filippone01@unipa.it}}

	\medskip\noindent
	{{Mario Galici}
		\\
		\footnotesize{Dipartimento di Matematica e Informatica}\\
		\footnotesize{Università degli Studi di Palermo, Via Archirafi 34, 90123 Palermo, Italy}\\
		\footnotesize{mario.galici@unipa.it}}

	\begin{abstract}
		The concept of \emph{Schreier extensions} of loops was introduced in the general case in \cite{NagyStrambach_Schreier} and, more recently, it has been explored in the context of Steiner loops in \cite{FFG}. In the latter case, it gives a powerful method for constructing Steiner triple systems containing Veblen points. Counting all Steiner triple systems of order $v$ is an open problem for $v>21$. In this paper, we investigate the number of Steiner triple systems of order $19$, $27$ and $31$ containing Veblen points and we present some examples. 

	\end{abstract}

		\section{Introduction}
		
		Classifying all Steiner triple systems of a given order $v$ ($\mathrm{STS}(v)$s for short) is a wild problem. The last full classification was done by P. Kaski and P. R. J. {\"O}sterg{\aa}rd in \cite{STS19}, in which they determined the $11,084,874,829$ non-isomorphic $\mathrm{STS}(19)$s. Attempting to classify all Steiner triple systems for the next admissible order, which is $21$, appears currently unfeasible. However, in \cite{enumeratiostswithsub} and \cite{sts21withsub}, the authors classified $\mathrm{STS}(21)$s containing subsystems of order $7$ and $9$, and also gave an estimation of the total number of all $\mathrm{STS}(21)$s. Due to this estimation, instead of achieving a complete classification, {\"O}sterg{\aa}rd and Heinlein in  \cite{STS21HeinleinOstergard} focused on the more manageable task of finding the number of the isomorphism classes of $\mathrm{STS}(21)$s, which is $14,796,207,517,873,771$.

		For this reason, instead of a complete classification, it is better to look at Steiner triple systems presenting some regular structures, for example having some given subsystems. In particular, we focus our investigation on the number of $\mathrm{STS}$s containing particular points, known as \emph{Veblen points}. A Veblen point has the property that any two distinct triples through it generate a Pasch configuration, hence an $\mathrm{STS}(7)$. In this sense, Steiner triple systems containing Veblen points are close to projective $\mathrm{STS}$s. The latter are indeed characterized by the fact every element is a Veblen point \cite[Th. 8.15]{ColbournRosa}. For this reason, we can say that Steiner triple systems with Veblen points preserve a regularity resembling the structure of a projective geometry. In \cite{FFG}, the authors provide a theoretical algebraic technique for constructing Steiner triple systems containing Veblen points by using methods from the theory of loop extensions. We will use this approach to investigate the cases of order $19$, $27$ and $31$, respectively. In \Cref{appendix}, we explain how we applied this method to achieve our results. Before going into the details, let us briefly recall some preliminary facts. 
		
		\subsection{Preliminaries}
		
		A Steiner triple system consists of a pair $(\mathcal{S},\mathcal{T})$, where $\mathcal{S}$ is a set of points and $\mathcal{T}$ is a family of triples of elements of $\mathcal{S}$, with the property that for any two different points in $\mathcal{S}$ there exists precisely one triple in $\mathcal{T}$ containing them. Generally, we will denote a Steiner triple system just by its set of points. A Steiner triple system of order $v$ exists if and only if $v\equiv 1$ or $3$ $\mathrm{mod} \ 6$, and in this case $v$ is said to be \emph{admissible}. Classical examples of Steiner triple systems include the \emph{projective} and \emph{affine} $\mathrm{STS}$s. The former are the point-line designs of an $n$-dimensional projective space $\mathrm{PG}(n,2)$ over the field $\mathrm{GF(2)}$. The latter are the point-line designs of an $n$-dimensional affine space $\mathrm{AG}(n,3)$ over the field $\mathrm{GF(3)}$. In both cases, the points of the Steiner triple system are the points of the space, and the triples are precisely the lines.
		
		In this paper we also deal with the concept of \emph{loops}, which has its own great importance in the field of abstract algebra, and also a significant relation to Steiner triple systems. A \emph{loop} is a set $L$ equipped with a binary operation $x\cdot y$ and an identity element $\Omega\in L$ such that the equations $x\cdot a=b$ and $a\cdot y=b$ have unique solutions $x,y\in L$. A loop does not need to be associative, and if this is the case, then it is a group. The reader is referred to \cite{bruckbook} or \cite{pflugfelder} for more insights  about loop theory.
		
		Any Steiner triple system $\mathcal{S}$ can be associated with a loop called \emph{Steiner loop}. This loop is defined over the set $\mathcal{L_S}:=\mathcal{S}\cup\{\Omega\}$, where $\Omega$ is an additional point, and the operation is described as follows:
		\begin{itemize}
		\item for any $x \in \mathcal{L_S}$, $x \cdot x=\Omega$ and $x \cdot \Omega=\Omega \cdot x=x$;
		\item for any distinct $x,y \in \mathcal{S}$, $x \cdot y=z$, where $z$ is the third point in the triple of  $\mathcal{S}$ containing both $x$ and $y$, or equivalently $x\cdot y \cdot z=\Omega$ for any triple $\{x,y,z\}$ of $\mathcal{S}$.
		\end{itemize}
		By definition, $\mathcal{L_S}$ is a commutative loop of exponent $2$ with identity element $\Omega$. Furthermore, $\mathcal{L_S}$ is a group, more specifically an elementary abelian $2$-group, exactly when the corresponding Steiner triple system $\mathcal{S}$ is projective.
		
		A point $x$ in an $\mathrm{STS}$ $\mathcal{S}$ is a \emph{Veblen point} if whenever $\{x,a,b\}$, $\{x,c,d\}$, $\{t,a,c\}$ are triples of $\mathcal{S}$, also $\{t,b,d\}$ is a triple of $\mathcal{S}$. In other words, two different triples through $x$ generate a Pasch configuration, as in \Cref{fig:pasch_antipasch_conf}.

		\begin{figure}[H]
		\begin{center}
			\begin{tikzpicture}[scale=0.30] 
				\draw[line width=1] (-6,0) -- (0,10);
				\draw[line width=1] (6,0) -- (0,10);
				\draw[line width=1] (-6,0) -- (3,5);
				\draw[line width=1] (6,0) -- (-3,5);

				\fill (-6,0) circle (10pt);
				\fill (6,0) circle (10pt);
				\fill (0,10) circle (10pt);
				\fill (3,5) circle (10pt);
				\fill (0,3.38) circle (10pt);
				\fill (-3,5) circle (10pt);
				
				\node at (0,11) {$x$};
				\node at (0,2) {$t$};
				\node at (-7,0) {$b$};
				\node at (7,0) {$c$};
				\node at (-4,5) {$a$};
				\node at (4,5) {$d$};
			\end{tikzpicture}
		\end{center}
		\caption{A Pasch configuration through a Veblen point $x$.}
		\label{fig:pasch_antipasch_conf}
		\end{figure}
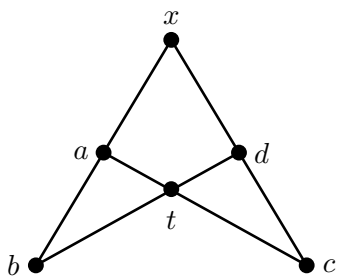
		
		As noted in \cite{FFG}, a point $x\in\mathcal{S}$ is a Veblen point if  and only if it is a non-trivial central element of $\mathcal{L_S}$. It is important to recall here that the center of a loop consists of all the elements which associate and commute with every other element within the loop. 
		
		The method presented in \cite{FFG} for constructing Steiner triple systems with Veblen points is based on the concept of the so-called \emph{Schreier extensions} of Steiner loops, introduced in a more general case in  \cite{NagyStrambach_Schreier}. Let us consider
		\begin{itemize}
			\item an elementary abelian $2$-group $(\mathcal{L_N},+)$ with identity element $\Omega$,
			
			\item a Steiner loop $(\mathcal{L_Q},\cdot)$ with identity element $\bar\Omega$,
			
			\item  a symmetric function $f\colon \mathcal{L_Q}\times\mathcal{L_Q}\to\mathcal{L_N}$  such that $f(P,\bar\Omega)=f(P,P)=\Omega$ for every $P\in\mathcal{L_Q}$ and 
			\begin{equation}\label{consttriples}
			f(P,Q)=f(P,R)=f(Q,R) \quad \text{whenever} \quad \{P,Q,R\} \text{ is a triple of } \mathcal{Q}.
			\end{equation}
		\end{itemize}
		The following operation,
		\begin{equation*}
			(P,x)\circ(Q,y)=(PQ,x+y+f(P,Q)),
		\end{equation*}
		makes the set $\mathcal{L_Q}\times\mathcal{L_N}$ a Steiner loop $\mathcal{L_S}$ with identity element $(\bar\Omega, \Omega)$. This new Steiner loop contains $\mathcal{L_N}$ as a central subloop, and the corresponding quotient loop is isomorphic to $\mathcal{L_Q}$. In other words, $\mathcal{L_S}$ is a central extension of $\mathcal{L_N}$ by $\mathcal{L_Q}$, known as \emph{Schreier extension}. The function $f$ is the \emph{factor system} of the extension. Since $\mathcal{L_N}$ is contained in the center of $\mathcal{L_S}$, every element of $\mathcal{N}$ is a Veblen point of $\mathcal{S}$. We remark here that every loop with a central subgroup can be described as a Schreier extension of this central subgroup by the corresponding quotient loop (cf. \cite[p. 334]{Bruck_contributions}). 
		
		We denote by $\mathrm{Ext_S}(\mathcal{L_N},\mathcal{L_Q})$ the  group of all the factor systems. From \Cref{consttriples}, we have that the order of $\mathrm{Ext_S}(\mathcal{L_N},\mathcal{L_Q})$  is $|\mathcal{L_N}|^{b}$,
		where $b$ is the number of triples of $\mathcal{Q}$.
		
		Two Schreier extensions of $\mathcal{L_N}$ by $\mathcal{L_Q}$, denoted as $\mathcal{L}_{\mathcal{S}_1}$ and $\mathcal{L}_{\mathcal{S}_2}$, are said \emph{equivalent} if there exists an isomorphism of loops $\phi\colon \mathcal{L}_{\mathcal{S}_1}\longrightarrow\mathcal{L}_{\mathcal{S}_2}$ which fixes point-wise $\mathcal{L_N}$ and induces the identity also on $\mathcal{L_Q}$. They are said \emph{isomorphic} if there exists an isomorphism of loops $\phi\colon \mathcal{L}_{\mathcal{S}_1}\longrightarrow\mathcal{L}_{\mathcal{S}_2}$ which fixes globally $\mathcal{L_N}$. In these cases, we also refer to the corresponding factor systems as equivalent or, respectively, isomorphic.

		As a characterization, in \cite{FFG} it is shown that two factor systems $f_1$ and $f_2$ are equivalent if, and only if, they differ by a \emph{co-boundary}, that is, $f_1=f_2+\delta^1\varphi$, for some function $\varphi\colon\mathcal{L_Q}\to\mathcal{L_N}$ with $\varphi(\bar\Omega)=\Omega$, where $(\delta^1\varphi)(P,Q)=\varphi(P)+\varphi(Q)-\varphi(PQ)$.
		Hence, if we denote by $\mathrm{B}^2(\mathcal{L_N},\mathcal{L_Q})$ the set of the co-boundaries, the number of non-equivalent Schreier extensions of $\mathcal{L_N}$ by $\mathcal{L_Q}$ is
		$$\frac{|\mathrm{Ext_S}(\mathcal{L_N},\mathcal{L_Q})|}{|\mathrm{B}^2(\mathcal{L_N},\mathcal{L_Q})|}.$$

		Naturally, equivalent Schreier extensions are  also isomorphic, but the converse is not true in general. However, reducing up to equivalence is useful in order to easily characterize isomorphism of extensions. Indeed, when $\mathcal{L_N}$ is the whole center of $\mathcal{L_S}$, the isomorphism classes of all the factor systems are exactly the orbits of the action of the group $\mathrm{Aut}(\mathcal{L_N})\times\mathrm{Aut}(\mathcal{L_Q})$ on the set of non-equivalent factor systems defined by
		\begin{equation}
			(\alpha,\beta)(f)=\alpha f \beta^{-1},
		\end{equation}
		 where $(f\beta^{-1})(P,Q):=f\left(\beta^{-1}(P),\beta^{-1}(Q)\right)$.
	
		By definition, two isomorphic Schreier extensions give in turn two isomorphic Steiner loops, and consequently, two isomorphic Steiner triple systems. However, it is important to notice that the converse is not always true. Since the isomorphism between the Steiner loops might not preserve $\mathcal{L_N}$, it is possible to have two isomorphic Steiner loops arising from non-isomorphic Schreier extensions. This situation may occur when the Steiner triple systems produced by the two Schreier extensions have additional Veblen points not contained in $\mathcal{N}$, that is, the case where $\mathcal{L_N}$ is not the whole center. In \cite{FFG}, the author provide a criterion to determine whether this is the case.
		
		\begin{remark}[see \cite{FFG}]\label{rmk:furtherveblen}
			Let the Steiner loop $\mathcal{L_S}$ be a Schreier extension of $\mathcal{L_N}$ by $\mathcal{L_Q}$ with factor system $f$.  The Steiner triple system $\mathcal{S}$ has a Veblen point $(P,x)$ not contained in $\mathcal{N}$ if and only if $P$ is itself a Veblen point of $\mathcal{Q}$ and 
			\begin{equation} \label{eq:furtherveblen}
				f(P,Q)+f(PQ,R)=f(Q,R)+f(P,QR),
			\end{equation}
			for every $Q,R\in\mathcal{Q}$. The condition \eqref{eq:furtherveblen} reflects the centrality of the element $(P,x)$ in $\mathcal{L_S}$.
		\end{remark}

		In \cite{FFG} the authors prove that there exists an $\mathrm{STS}(v)$ with (at least) $2^c-1<v$ Veblen points if and only if 
		\begin{equation}\label{eq:numvebpoints}
			\frac{v+1}{2^c}\equiv 2,4 \pmod 6.
		\end{equation}
		Thus, the first few orders $v>7$ of Steiner triple systems which can have Veblen points are $15$, $19$, $27$ and $31$.  Indeed, among the $80$ non-isomorphic Setiner triple systems of order $15$ (see \cite[Table II.1.28]{Handbook}) only two of them have Veblen points: $\mathrm{STS}(15)$ $\# 1$, which is $\mathrm{PG}(3,2)$ and, therefore, has $15$ Veblen points, and $\mathrm{STS}(15)$ $\# 2$ which has precisely one Veblen point. Here, we recall a final result that we will use for our classification.
		
		\begin{theorem}[see \cite{FFG}]\label{theoindex4}
			If a Schreier extension of Steiner loops
			\begin{equation}
				\Omega\longrightarrow \mathcal{L_N} \longrightarrow \mathcal{L_S} \longrightarrow \mathcal{L_Q} \longrightarrow \bar\Omega,
			\end{equation}
			has index at most $4$, then the resulting Steiner triple system $\mathcal{S}$ is projective.
		\end{theorem}
		
		This means that a non-projective Steiner triple system of order $v$ can have at most $\frac{v-7}{8}$ Veblen points. 
		
		In this paper we are going to study the cases of $\mathrm{STS}$s of order $19$, $27$ and $31$. In \Cref{appendix}, we give some details of the implementation used to obtain our results.

		\section{Steiner triple systems of order $19$}
		
		By \Cref{eq:numvebpoints}, the number of Veblen points of a Steiner triple system of order $19$ can be at most $1$.

		\begin{theorem}
			Among the $11,084,874,829$ non-isomorphic $\mathrm{STS}(19)$s, there are only $3$ Steiner triple systems with (exactly) one Veblen point.
		\end{theorem}

		\begin{proof}
			If an $\mathrm{STS}(19)$ $\mathcal{S}$ has one Veblen point, then this point is the only non-trivial central element of the Steiner loop $\mathcal{L_S}$. Hence, we can obtain $\mathcal{L_S}$ as a Schreier extension of its center $\mathcal{L_N}$, which is the group of order $2$, by the unique Steiner loop $\mathcal{L_Q}$ of order $10$ corresponding to the $\mathrm{STS}(9)$. Since the order of $\mathcal{L_N}$ is $2$ and $\mathcal{Q}$ has $12$ triples, the total number of possible factor systems in $\mathrm{Ext}_S(\mathcal{L_N},\mathcal{L_Q})$ is $2^{12}=4096$. 
			
			Now, we apply the algorithms described in \Cref{sec:descriptionalgorithms}.
			The set of co-boundaries $\mathrm{B}^2$ has order $2^9=512$. Consequently, the number of non-equivalent extensions is $8$. We note now that, since $\mathrm{Aut}(\mathcal{L_N})$ is trivial, the action of the group $\mathrm{Aut}(\mathcal{L_N})\times\mathrm{Aut}(\mathcal{L_Q})$ can be reduced to the action of just $\mathrm{Aut}(\mathcal{L_Q})=\mathrm{Aff}(2,3)$, which has order $432$. The $8$ equivalence classes of factor systems are divided under this action into $3$ orbits, each of which represents an isomorphism class of extensions of $\mathcal{L_N}$ by $\mathcal{L_Q}$. Since $\mathcal{L_N}$ is the center of $\mathcal{L_S}$, these orbits correspond to the isomorphism classes of $\mathrm{STS}(19)$s with precisely one Veblen point. 
		\end{proof}

		In order to provide a full description of the three Steiner triple systems of order $19$ containing one Veblen point, let represent the $\mathrm{STS}(9)$ as the affine plane $\mathrm{AG}(2,3)$, with points and lines as in \Cref{fig:sts9chap3}.
		
		\begin{figure}[H]
			\begin{center}
				\begin{tikzpicture}[line cap=round,line join=round,>=triangle 45,x=1cm,y=1cm,scale=1.4]
					\draw [line width=1pt,color=gray] (-1,-1)-- (1,1);
					\draw [line width=1pt,color=gray] (-1,-1)-- (1,-1);
					\draw [line width=1.5pt,color=gray] (1,-1)-- (1,1);
					\draw [line width=1pt,color=gray] (1,1)-- (-1,1);
					\draw [line width=1pt,color=gray] (-1,1)-- (-1,-1);
					\draw [line width=1pt,color=gray] (-1,0)-- (1,0);
					\draw [line width=1pt,color=gray] (0,1)-- (0,-1);
					\draw [line width=1pt,color=gray] (1,-1)-- (-1,1);
					\draw [line width=1pt,color=gray] (0,-1)-- (1,0);
					\draw [line width=1pt,color=gray] (1,0)-- (0,1);
					\draw [line width=1pt,color=gray] (0,1)-- (-1,0);
					\draw [line width=1pt,color=gray] (-1,0)-- (0,-1);
					\draw [shift={(-1.699214041911079,-0.8152542693528347)},line width=1pt,color=gray]  plot[domain=2.3117548637419008:4.2939481603980205,variable=\t]({-0.8196591288696354*3.96958988026553*cos(\t r)+0.5728515623271623*1.1900139979490725*sin(\t r)},{-0.5728515623271623*3.96958988026553*cos(\t r)+-0.8196591288696354*1.1900139979490725*sin(\t r)});
					\draw [shift={(1.699214041911079,-0.8152542693528345)},line width=1pt,color=gray]  plot[domain=2.3117548637419008:4.293948160398021,variable=\t]({0.8196591288696355*3.96958988026553*cos(\t r)+-0.5728515623271622*1.1900139979490725*sin(\t r)},{-0.5728515623271622*3.96958988026553*cos(\t r)+-0.8196591288696355*1.1900139979490725*sin(\t r)});
					\draw [shift={(-1.699214041911079,0.8152542693528347)},line width=1pt,color=gray]  plot[domain=2.3117548637419008:4.2939481603980205,variable=\t]({-0.8196591288696354*3.96958988026553*cos(\t r)+0.5728515623271623*1.1900139979490725*sin(\t r)},{0.5728515623271623*3.96958988026553*cos(\t r)+0.8196591288696354*1.1900139979490725*sin(\t r)});
					\draw [shift={(1.699214041911079,0.815254269352835)},line width=1pt,color=gray]  plot[domain=2.3117548637419008:4.2939481603980205,variable=\t]({0.8196591288696353*3.96958988026553*cos(\t r)+-0.5728515623271624*1.1900139979490725*sin(\t r)},{0.5728515623271624*3.96958988026553*cos(\t r)+0.8196591288696353*1.1900139979490725*sin(\t r)});
						\draw [fill=black] (-1,1) circle (2pt);
						\draw [fill=black] (0,1) circle (2pt);
						\draw [fill=black] (1,1) circle (2pt);
						\draw [fill=black] (-1,0) circle (2pt);
						\draw [fill=black] (0,0) circle (2pt);
						\draw [fill=black] (1,0) circle (2pt);
						\draw [fill=black] (-1,-1) circle (2pt);
						\draw [fill=black] (0,-1) circle (2pt);
						\draw [fill=black] (1,-1) circle (2pt);

						\node at (-1.20,1.2) {$P_1$};
						\node at (0,1.2) {$P_2$};
						\node at (1.20,1.2) {$P_3$};
						\node at (-1.25,0) {$P_4$};
						\node at (0.35,0.12) {$P_5$};
						\node at (1.25,0) {$P_6$};
						\node at (-1.20,-1.25) {$P_7$};
						\node at (0,-1.25) {$P_8$};
						\node at (1.20,-1.25) {$P_9$};
				\end{tikzpicture}\caption{The $\mathrm{STS}(9)$ as the affine plane $\mathrm{AG}(2,3)$.} \label{fig:sts9chap3}
			\end{center}
		\end{figure}
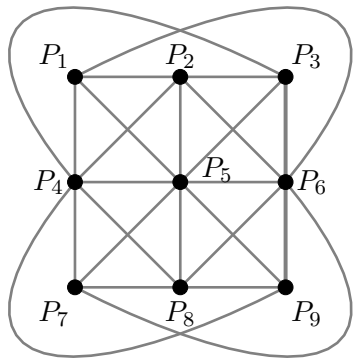

		Let us denote the $3$ non-isomorphic $\mathrm{STS}(19)$s containing one Veblen point with $\mathcal{S}_i$, $i=0,1,2$. The corresponding Steiner loops $\mathcal{L}_{\mathcal{S}_i}$ are given as Schreier extensions of $\mathcal{L_N}=\{\Omega, 1\}$ by $\mathcal{L_Q}$, with the following factor systems, respectively:
		\begin{align*}
		 &f_0, \quad \text{the null function, i.e., the function mapping every pair into $\Omega$;}\\
		 &f_1, \quad \text{defined by } f_1(P_3,P_6)=f_1(P_3,P_9)=f_1(P_6,P_9)=1, \quad \text{and $\Omega$ elsewhere;}\\
		 &f_2, \quad \text{defined by } f_2(P_3,P_6)=f_2(P_3,P_9)=f_2(P_6,P_9)=1, \ f_2(P_7,P_8)=f_2(P_7,P_9)=f_2(P_8,P_9)=1, \\ &\quad \quad \ \text{and $\Omega$ elsewhere.}
		\end{align*}
		
		 \Cref{fig:f1f2} provides a visual representation of the non-trivial factor systems $f_1$ and $f_2$, where the triples of the $\mathrm{STS}(9)$ $\mathcal{Q}$ in which $f_i$ is non-zero, $i=1,2$, are drawn in red.
		
		\begin{figure}[H]
			\begin{center}
				\begin{tikzpicture}[line cap=round,line join=round,>=triangle 45,x=1cm,y=1cm]
					\draw [line width=0.5pt,color=gray] (-1,-1)-- (1,1);
					\draw [line width=0.5pt,color=gray] (-1,-1)-- (1,-1);
					\draw [line width=1.5pt,color=red] (1,-1)-- (1,1);
					\draw [line width=0.5pt,color=gray] (1,1)-- (-1,1);
					\draw [line width=0.5pt,color=gray] (-1,1)-- (-1,-1);
					\draw [line width=0.5pt,color=gray] (-1,0)-- (1,0);
					\draw [line width=0.5pt,color=gray] (0,1)-- (0,-1);
					\draw [line width=0.5pt,color=gray] (1,-1)-- (-1,1);
					\draw [line width=0.5pt,color=gray] (0,-1)-- (1,0);
					\draw [line width=0.5pt,color=gray] (1,0)-- (0,1);
					\draw [line width=0.5pt,color=gray] (0,1)-- (-1,0);
					\draw [line width=0.5pt,color=gray] (-1,0)-- (0,-1);
					\draw [shift={(-1.699214041911079,-0.8152542693528347)},line width=0.5pt,color=gray]  plot[domain=2.3117548637419008:4.2939481603980205,variable=\t]({-0.8196591288696354*3.96958988026553*cos(\t r)+0.5728515623271623*1.1900139979490725*sin(\t r)},{-0.5728515623271623*3.96958988026553*cos(\t r)+-0.8196591288696354*1.1900139979490725*sin(\t r)});
					\draw [shift={(1.699214041911079,-0.8152542693528345)},line width=0.5pt,color=gray]  plot[domain=2.3117548637419008:4.293948160398021,variable=\t]({0.8196591288696355*3.96958988026553*cos(\t r)+-0.5728515623271622*1.1900139979490725*sin(\t r)},{-0.5728515623271622*3.96958988026553*cos(\t r)+-0.8196591288696355*1.1900139979490725*sin(\t r)});
					\draw [shift={(-1.699214041911079,0.8152542693528347)},line width=0.5pt,color=gray]  plot[domain=2.3117548637419008:4.2939481603980205,variable=\t]({-0.8196591288696354*3.96958988026553*cos(\t r)+0.5728515623271623*1.1900139979490725*sin(\t r)},{0.5728515623271623*3.96958988026553*cos(\t r)+0.8196591288696354*1.1900139979490725*sin(\t r)});
					\draw [shift={(1.699214041911079,0.815254269352835)},line width=0.5pt,color=gray]  plot[domain=2.3117548637419008:4.2939481603980205,variable=\t]({0.8196591288696353*3.96958988026553*cos(\t r)+-0.5728515623271624*1.1900139979490725*sin(\t r)},{0.5728515623271624*3.96958988026553*cos(\t r)+0.8196591288696353*1.1900139979490725*sin(\t r)});
					\begin{scriptsize}
						\draw [fill=black] (-1,1) circle (2pt);
						\draw [fill=black] (0,1) circle (2pt);
						\draw [fill=black] (1,1) circle (2pt);
						\draw [fill=black] (-1,0) circle (2pt);
						\draw [fill=black] (0,0) circle (2pt);
						\draw [fill=black] (1,0) circle (2pt);
						\draw [fill=black] (-1,-1) circle (2pt);
						\draw [fill=black] (0,-1) circle (2pt);
						\draw [fill=black] (1,-1) circle (2pt);

						\node at (-1.25,1.2) {$P_1$};
						\node at (-1.25,0) {$P_4$};
						\node at (-1.25,-1.2) {$P_7$};
						\node at (0,1.2) {$P_2$};
						\node at (0,-1.2) {$P_8$};
						\node at (1.25,1.2) {$P_3$};
						\node at (1.25,0) {$P_6$};
						\node at (1.25,-1.2) {$P_9$};
						\node at (0.15,0.20) {$P_5$};
					\end{scriptsize}
				\end{tikzpicture}
				\hspace{2cm}
				\begin{tikzpicture}[line cap=round,line join=round,>=triangle 45,x=1cm,y=1cm]
					\draw [line width=0.5pt,color=gray] (-1,-1)-- (1,1);
					\draw [line width=1.5pt,color=red] (-1,-1)-- (1,-1);
					\draw [line width=1.5pt,color=red] (1,-1)-- (1,1);
					\draw [line width=0.5pt,color=gray] (1,1)-- (-1,1);
					\draw [line width=0.5pt,color=gray] (-1,1)-- (-1,-1);
					\draw [line width=0.5pt,color=gray] (-1,0)-- (1,0);
					\draw [line width=0.5pt,color=gray] (0,1)-- (0,-1);
					\draw [line width=0.5pt,color=gray] (1,-1)-- (-1,1);
					\draw [line width=0.5pt,color=gray] (0,-1)-- (1,0);
					\draw [line width=0.5pt,color=gray] (1,0)-- (0,1);
					\draw [line width=0.5pt,color=gray] (0,1)-- (-1,0);
					\draw [line width=0.5pt,color=gray] (-1,0)-- (0,-1);
					\draw [shift={(-1.699214041911079,-0.8152542693528347)},line width=0.5pt,color=gray]  plot[domain=2.3117548637419008:4.2939481603980205,variable=\t]({-0.8196591288696354*3.96958988026553*cos(\t r)+0.5728515623271623*1.1900139979490725*sin(\t r)},{-0.5728515623271623*3.96958988026553*cos(\t r)+-0.8196591288696354*1.1900139979490725*sin(\t r)});
					\draw [shift={(1.699214041911079,-0.8152542693528345)},line width=0.5pt,color=gray]  plot[domain=2.3117548637419008:4.293948160398021,variable=\t]({0.8196591288696355*3.96958988026553*cos(\t r)+-0.5728515623271622*1.1900139979490725*sin(\t r)},{-0.5728515623271622*3.96958988026553*cos(\t r)+-0.8196591288696355*1.1900139979490725*sin(\t r)});
					\draw [shift={(-1.699214041911079,0.8152542693528347)},line width=0.5pt,color=gray]  plot[domain=2.3117548637419008:4.2939481603980205,variable=\t]({-0.8196591288696354*3.96958988026553*cos(\t r)+0.5728515623271623*1.1900139979490725*sin(\t r)},{0.5728515623271623*3.96958988026553*cos(\t r)+0.8196591288696354*1.1900139979490725*sin(\t r)});
					\draw [shift={(1.699214041911079,0.815254269352835)},line width=0.5pt,color=gray]  plot[domain=2.3117548637419008:4.2939481603980205,variable=\t]({0.8196591288696353*3.96958988026553*cos(\t r)+-0.5728515623271624*1.1900139979490725*sin(\t r)},{0.5728515623271624*3.96958988026553*cos(\t r)+0.8196591288696353*1.1900139979490725*sin(\t r)});
					\begin{scriptsize}
						\draw [fill=black] (-1,1) circle (2pt);
						\draw [fill=black] (0,1) circle (2pt);
						\draw [fill=black] (1,1) circle (2pt);
						\draw [fill=black] (-1,0) circle (2pt);
						\draw [fill=black] (0,0) circle (2pt);
						\draw [fill=black] (1,0) circle (2pt);
						\draw [fill=black] (-1,-1) circle (2pt);
						\draw [fill=black] (0,-1) circle (2pt);
						\draw [fill=black] (1,-1) circle (2pt);

						\node at (-1.25,1.2) {$P_1$};
						\node at (-1.25,0) {$P_4$};
						\node at (-1.25,-1.2) {$P_7$};
						\node at (0,1.2) {$P_2$};
						\node at (0,-1.2) {$P_8$};
						\node at (1.25,1.2) {$P_3$};
						\node at (1.25,0) {$P_6$};
						\node at (1.25,-1.2) {$P_9$};
						\node at (0.15,0.20) {$P_5$};
					\end{scriptsize}
				\end{tikzpicture}\caption{The factor systems $f_1$ (left) and $f_2$ (right).}\label{fig:f1f2}
			\end{center}
		\end{figure}
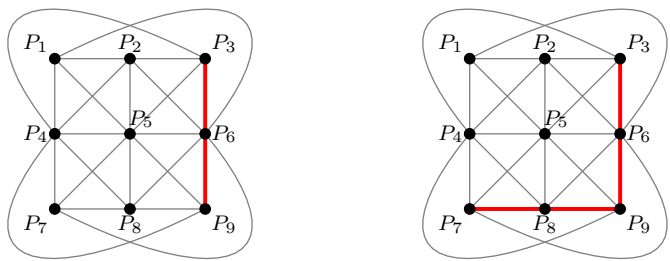
		
		Now we describe explicitly the points and triples of each of these Steiner triple systems $\mathcal{S}_i$, $i=1,2,3$.	Since the factor system $f_0$ is the null function, in the loop $\mathcal{L}_{\mathcal{S}_0}$ the operation is simply described by
		\begin{equation*}
			(P_i,x)\circ(P_j,y)=(P_iP_j,x+y).
		\end{equation*}
		Hence the triples are given by the families
		\begin{align*}
			&\bigl\{  (P_i,\Omega), \ (P_j,\Omega), \ (P_iP_j,\Omega) \mid P_i,P_j\in\mathcal{Q}, \ i\neq j \bigr\}, \\
			&\bigl\{  (P_i,1), \ (P_j,1), \ (P_iP_j,\Omega) \mid P_i,P_j\in\mathcal{Q}, \ i\neq j   \bigr\}, \\
			& \bigl\{  (P_i,1), \ (P_i,\Omega), \ (\Omega,1) \mid P_i\in\mathcal{Q}\bigr\}.
		\end{align*}
		We note that the first family has $12$ triples, the second one has $36$ and the third one has $9$, for a total of $57$, which is the exact number of triples in an $\mathrm{STS}(19)$.
		Renaming the elements as follows,
		\begin{align*}
			0&:=(\bar \Omega, 1) && \text{(the Veblen point),}\\
			i&:=(P_i,\Omega), && i=1,\dots,9	\\
			i+9&:=(P_i,1), && i=1,\dots,9
		\end{align*}
		we obtain the representation of the $\mathrm{STS}(19)$ $\mathcal{S}_0$ as in \Cref{tab:sts19veblen0}, where the columns are the triples of $\mathcal{S}_0$.

	\begin{table}[H]
		\centering
		\begin{tabular}{*{19}{c}}
			\hline
			$0 \!$ & $\! 0 \!$&$\! 0 \!$&$\! 0 \!$&$\! 0 \!$&$\! 0 \!$&$\! 0 \!$&$\! 0 \!$&$\! 0 \!$&$\! 1 \!$&$\! 1 \!$&$\! 1 \!$&$\! 1 \!$&$\! 1 \!$&$\! 1 \!$&$\! 1 \!$&$\! 1 \!$&$\! 2 \!$&$\! 2$ \\
			$1 \!$&$\! 2 \!$&$\! 3 \!$&$\! 4 \!$&$\! 5 \!$&$\! 6 \!$&$\! 7 \!$&$\! 8 \!$&$\! 9 \!$&$\! 2 \!$&$\! 4 \!$&$\! 5 \!$&$\! 6 \!$&$\! 11 \!$&$\! 13 \!$&$\! 14 \!$&$\! 15 \!$&$\! 4 \!$&$\! 5$ \\
			$10 \!$&$\! 11 \!$&$\! 12 \!$&$\! 13 \!$&$\! 14 \!$&$\! 15 \!$&$\! 16 \!$&$\! 17 \!$&$\! 18 \!$&$\! 3 \!$&$\! 7 \!$&$\! 9 \!$&$\! 8 \!$&$\! 12 \!$&$\! 16 \!$&$\! 18 \!$&$\! 17 \!$&$\! 9 \!$&$\! 8$ \\
			\hline
			\\ 
			\hline
			$2 \!$&$\! 2 \!$&$\! 2 \!$&$\! 2 \!$&$\! 2 \!$&$\! 3 \!$&$\! 3 \!$&$\! 3 \!$&$\! 3 \!$&$\! 3 \!$&$\! 3 \!$&$\! 3 \!$&$\! 4 \!$&$\! 4 \!$&$\! 4 \!$&$\! 4 \!$&$\! 4 \!$&$\! 5 \!$&$\! 5$ \\
			$6 \!$&$\! 10 \!$&$\! 13 \!$&$\! 14 \!$&$\! 15 \!$&$\! 4 \!$&$\! 5 \!$&$\! 6 \!$&$\! 10 \!$&$\! 13 \!$&$\! 14 \!$&$\! 15 \!$&$\! 5 \!$&$\! 10 \!$&$\! 11 \!$&$\! 12 \!$&$\! 14 \!$&$\! 10 \!$&$\! 11$ \\
			$7 \!$&$\! 12 \!$&$\! 18 \!$&$\! 17 \!$&$\! 16 \!$&$\! 8 \!$&$\! 7 \!$&$\! 9 \!$&$\! 11 \!$&$\! 17 \!$&$\! 16 \!$&$\! 18 \!$&$\! 6 \!$&$\! 16 \!$&$\! 18 \!$&$\! 17 \!$&$\! 15 \!$&$\! 18 \!$&$\! 17$ \\
			\hline
			\\ 
			\hline
			$5 \!$&$\! 5 \!$&$\! 6 \!$&$\! 6 \!$&$\! 6 \!$&$\! 6 \!$&$\! 7 \!$&$\! 7 \!$&$\! 7 \!$&$\! 7 \!$&$\! 7 \!$&$\! 8 \!$&$\! 8 \!$&$\! 8 \!$&$\! 8 \!$&$\! 9 \!$&$\! 9 \!$&$\! 9 \!$&$\! 9$ \\
			$12 \!$&$\! 13 \!$&$\! 10 \!$&$\! 11 \!$&$\! 12 \!$&$\! 13 \!$&$\! 8 \!$&$\! 10 \!$&$\! 11 \!$&$\! 12 \!$&$\! 17 \!$&$\! 10 \!$&$\! 11 \!$&$\! 12 \!$&$\! 16 \!$&$\! 10 \!$&$\! 11 \!$&$\! 12 \!$&$\! 16$ \\
			$16 \!$&$\! 15 \!$&$\! 17 \!$&$\! 16 \!$&$\! 18 \!$&$\! 14 \!$&$\! 9 \!$&$\! 13 \!$&$\! 15 \!$&$\! 14 \!$&$\! 18 \!$&$\! 15 \!$&$\! 14 \!$&$\! 13 \!$&$\! 18 \!$&$\! 14 \!$&$\! 13 \!$&$\! 15 \!$&$\! 17$ \\
			\hline
		\end{tabular}
		\caption{The $\mathrm{STS}(19)$ $\mathcal{S}_0$ with unique Veblen point 0.}\label{tab:sts19veblen0}
	\end{table}

	Since the factor system $f_1$ is zero everywhere except for any pair of points in the triple $\{P_3,P_6,P_9\}$ of the $\mathrm{STS}(9)$ $\mathcal{Q}$, the operation in $\mathcal{L}_{\mathcal{S}_1}$ is the same as in $\mathcal{L}_{\mathcal{S}_0}$ except for the following cases:
	\begin{align*}
		&(P_3, \Omega)\circ (P_6,\Omega)\circ(P_9,1)=(\bar\Omega,\Omega),\\
		&(P_3, \Omega)\circ (P_6,1)\circ(P_9,\Omega)=(\bar\Omega,\Omega),\\
		&(P_3, 1)\circ (P_6,\Omega)\circ(P_9,\Omega)=(\bar\Omega,\Omega),\\
		&(P_3, 1)\circ (P_6,1)\circ(P_9,1)=(\bar\Omega,\Omega).
	\end{align*}
	Hence the triples of $\mathcal{S}_1$ are obtained by those of $\mathcal{S}_0$ by performing the \emph{Pasch switch} (see, e.g., \cite{propertiesSTS19}) corresponding to substituting $\{3,6,9\}$, $\{3,15,18\}$, $\{6,12,18\}$ and $\{9,12,15\}$ with $\{3,6,18\}$, $\{3,15,9\}$, $\{6,9,12\}$ and $\{12,15,18\}$, as in \Cref{fig:paschswitch}. 
	
	\begin{figure}[H]
		\begin{center}
			\begin{tikzpicture}[scale=0.30] 
				\draw[line width=1] (-6,0) -- (0,10);
				\draw[line width=1] (6,0) -- (0,10);
				\draw[line width=1] (-6,0) -- (3,5);
				\draw[line width=1] (6,0) -- (-3,5);

				\fill[green] (-6,0) circle (10pt);
				\fill[blue] (6,0) circle (10pt);
				\fill (0,10) circle (10pt);
				\fill (3,5) circle (10pt);
				\fill (0,3.38) circle (10pt);
				\fill (-3,5) circle (10pt);
				
				\node at (0,11) {$3$};
				\node at (0,2) {$12$};
				\node at (-7,0) {$9$};
				\node at (7.25,0) {$18$};
				\node at (-4,5) {$6$};
				\node at (4.25,5) {$15$};
			\end{tikzpicture}
			\begin{tikzpicture}[scale=0.30] 
				\draw[->,decorate,decoration={snake,amplitude=.4mm,segment length=5mm,post length=1mm}] (-2,6) -- (2,6);
				\fill[white] (0,10) circle (10pt);
				\fill[white] circle (10pt);
			\end{tikzpicture}
			\begin{tikzpicture}[scale=0.30] 
				\draw[line width=1] (-6,0) -- (0,10);
				\draw[line width=1] (6,0) -- (0,10);
				\draw[line width=1] (-6,0) -- (3,5);
				\draw[line width=1] (6,0) -- (-3,5);

				\fill[blue] (-6,0) circle (10pt);
				\fill[green] (6,0) circle (10pt);
				\fill (0,10) circle (10pt);
				\fill (3,5) circle (10pt);
				\fill (0,3.38) circle (10pt);
				\fill (-3,5) circle (10pt);
				
				\node at (0,11) {$3$};
				\node at (0,2) {$12$};
				\node at (-7.25,0) {$18$};
				\node at (7,0) {$9$};
				\node at (-4,5) {$6$};
				\node at (4.25,5) {$15$};
			\end{tikzpicture}\caption{The Pasch switch transforming the triples of $\mathcal{S}_0$ into the triples of $\mathcal{S}_1$.}\label{fig:paschswitch}
		\end{center}
	\end{figure}
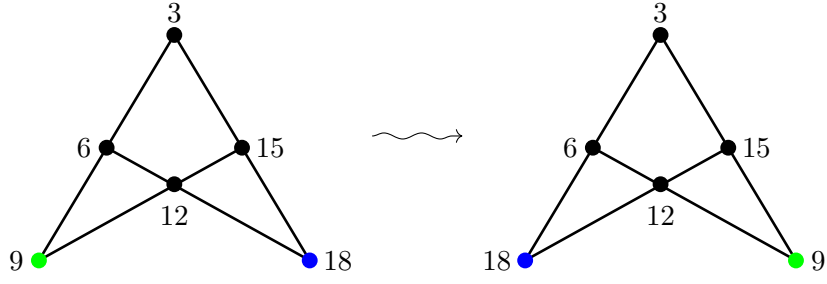

	Therefore, we find the representation of the $\mathrm{STS}(19)$ $\mathcal{S}_1$ shown in \Cref{tab:sts19veblen1}, where the columns are the triples of $\mathcal{S}_1$.
	
\begin{table}[H]
	\centering
	\begin{tabular}{*{19}{c}}
		\hline
		$\! 0 \!$ & $\! 0 \!$ & $\! 0 \!$ & $\! 0 \!$ & $\! 0 \!$ & $\! 0 \!$ & $\! 0 \!$ & $\! 0 \!$ & $\! 0 \!$ & $\! 1 \!$ & $\! 1 \!$ & $\! 1 \!$ & $\! 1 \!$ & $\! 1 \!$ & $\! 1 \!$ & $\! 1 \!$ & $\! 1 \!$ & $\! 2 \!$ & $\! 2 \!$ \\
		$\! 1 \!$ & $\! 2 \!$ & $\! 3 \!$ & $\! 4 \!$ & $\! 5 \!$ & $\! 6 \!$ & $\! 7 \!$ & $\! 8 \!$ & $\! 9 \!$ & $\! 2 \!$ & $\! 4 \!$ & $\! 5 \!$ & $\! 6 \!$ & $\! 11 \!$ & $\! 13 \!$ & $\! 14 \!$ & $\! 15 \!$ & $\! 4 \!$ & $\! 5 \!$ \\
		$\! 10 \!$ & $\! 11 \!$ & $\! 12 \!$ & $\! 13 \!$ & $\! 14 \!$ & $\! 15 \!$ & $\! 16 \!$ & $\! 17 \!$ & $\! 18 \!$ & $\! 3 \!$ & $\! 7 \!$ & $\! 9 \!$ & $\! 8 \!$ & $\! 12 \!$ & $\! 16 \!$ & $\! 18 \!$ & $\! 17 \!$ & $\! 9 \!$ & $\! 8 \!$ \\
		\hline
		\\ 
		\hline
		$\! 2 \!$ & $\! 2 \!$ & $\! 2 \!$ & $\! 2 \!$ & $\! 2 \!$ & $\! 3 \!$ & $\! 3 \!$ & $\! 3 \!$ & $\! 3 \!$ & $\! 3 \!$ & $\! 3 \!$ & $\! 3 \!$ & $\! 4 \!$ & $\! 4 \!$ & $\! 4 \!$ & $\! 4 \!$ & $\! 4 \!$ & $\! 5 \!$ & $\! 5 \!$ \\
		$\! 6 \!$ & $\! 10 \!$ & $\! 13 \!$ & $\! 14 \!$ & $\! 15 \!$ & $\! 4 \!$ & $\! 5 \!$ & $\! 6 \!$ & $\! 9 \!$ & $\! 10 \!$ & $\! 13 \!$ & $\! 14 \!$ & $\! 5 \!$ & $\! 10 \!$ & $\! 11 \!$ & $\! 12 \!$ & $\! 14 \!$ & $\! 10 \!$ & $\! 11 \!$ \\
		$\! 7 \!$ & $\! 12 \!$ & $\! 18 \!$ & $\! 17 \!$ & $\! 16 \!$ & $\! 8 \!$ & $\! 7 \!$ & $\! 18 \!$ & $\! 15 \!$ & $\! 11 \!$ & $\! 17 \!$ & $\! 16 \!$ & $\! 6 \!$ & $\! 16 \!$ & $\! 18 \!$ & $\! 17 \!$ & $\! 15 \!$ & $\! 18 \!$ & $\! 17 \!$ \\
		\hline
		\\ 
		\hline
		$\! 5 \!$ & $\! 5 \!$ & $\! 6 \!$ & $\! 6 \!$ & $\! 6 \!$ & $\! 6 \!$ & $\! 7 \!$ & $\! 7 \!$ & $\! 7 \!$ & $\! 7 \!$ & $\! 7 \!$ & $\! 8 \!$ & $\! 8 \!$ & $\! 8 \!$ & $\! 8 \!$ & $\! 9 \!$ & $\! 9 \!$ & $\! 9 \!$ & $\! 12 \!$ \\
		$\! 12 \!$ & $\! 13 \!$ & $\! 9 \!$ & $\! 10 \!$ & $\! 11 \!$ & $\! 13 \!$ & $\! 8 \!$ & $\! 10 \!$ & $\! 11 \!$ & $\! 12 \!$ & $\! 17 \!$ & $\! 10 \!$ & $\! 11 \!$ & $\! 12 \!$ & $\! 16 \!$ & $\! 10 \!$ & $\! 11 \!$ & $\! 16 \!$ & $\! 15 \!$ \\
		$\! 16 \!$ & $\! 15 \!$ & $\! 12 \!$ & $\! 17 \!$ & $\! 16 \!$ & $\! 14 \!$ & $\! 9 \!$ & $\! 13 \!$ & $\! 15 \!$ & $\! 14 \!$ & $\! 18 \!$ & $\! 15 \!$ & $\! 14 \!$ & $\! 13 \!$ & $\! 18 \!$ & $\! 14 \!$ & $\! 13 \!$ & $\! 17 \!$ & $\! 18 \!$ \\
		\hline
	\end{tabular}
	\caption{The $\mathrm{STS}(19)$ $\mathcal{S}_1$ with unique Veblen point 0.} \label{tab:sts19veblen1}
\end{table}

	Since the factor system $f_2$ coincides with $f_1$ everywhere except for any pair of points in the triple $\{P_7,P_8,P_9\}$ of the $\mathrm{STS}(9)$ $\mathcal{Q}$, the operation in $\mathcal{L}_{\mathcal{S}_2}$ is the same as in $\mathcal{L}_{\mathcal{S}_1}$ except for the following cases:
	\begin{align*}
		&(P_7, \Omega)\circ (P_8,\Omega)\circ(P_9,1)=(\bar\Omega,\Omega),\\
		&(P_7, \Omega)\circ (P_8,1)\circ(P_9,\Omega)=(\bar\Omega,\Omega),\\
		&(P_7, 1)\circ (P_8,\Omega)\circ(P_9,\Omega)=(\bar\Omega,\Omega),\\
		&(P_7, 1)\circ (P_8,1)\circ(P_9,1)=(\bar\Omega,\Omega).
	\end{align*}
	Hence the triples of $\mathcal{S}_2$ are obtained by those of $\mathcal{S}_1$ by the Pasch switch corresponding to substituting $\{7,8,9\}$, $\{7,17,18\}$, $\{8,16,18\}$ and $\{9,16,17\}$ with $\{7,8,18\}$, $\{7,9,17\}$, $\{8,9,16\}$ and $\{16,17,18\}$, as in \Cref{fig:paschswitch2}.
	
	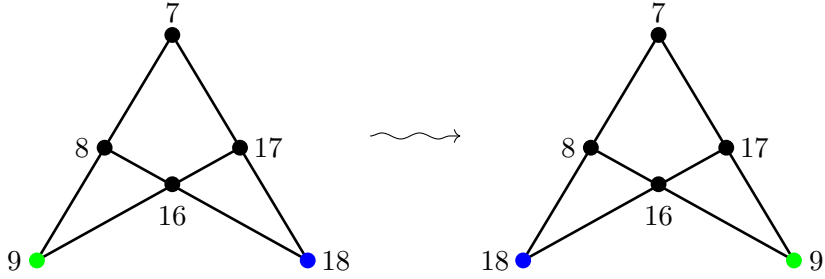
\begin{figure}[H]
		\begin{center}
			\begin{tikzpicture}[scale=0.30] 
				\draw[line width=1] (-6,0) -- (0,10);
				\draw[line width=1] (6,0) -- (0,10);
				\draw[line width=1] (-6,0) -- (3,5);
				\draw[line width=1] (6,0) -- (-3,5);

				\fill[green] (-6,0) circle (10pt);
				\fill[blue] (6,0) circle (10pt);
				\fill (0,10) circle (10pt);
				\fill (3,5) circle (10pt);
				\fill (0,3.38) circle (10pt);
				\fill (-3,5) circle (10pt);
				
				\node at (0,11) {$7$};
				\node at (0,2) {$16$};
				\node at (-7,0) {$9$};
				\node at (7.25,0) {$18$};
				\node at (-4,5) {$8$};
				\node at (4.25,5) {$17$};
			\end{tikzpicture}
			\begin{tikzpicture}[scale=0.30] 
				\draw[->,decorate,decoration={snake,amplitude=.4mm,segment length=5mm,post length=1mm}] (-2,6) -- (2,6);
				\fill[white] (0,10) circle (10pt);
				\fill[white] circle (10pt);
			\end{tikzpicture}
			\begin{tikzpicture}[scale=0.30] 
				\draw[line width=1] (-6,0) -- (0,10);
				\draw[line width=1] (6,0) -- (0,10);
				\draw[line width=1] (-6,0) -- (3,5);
				\draw[line width=1] (6,0) -- (-3,5);

				\fill[blue] (-6,0) circle (10pt);
				\fill[green] (6,0) circle (10pt);
				\fill (0,10) circle (10pt);
				\fill (3,5) circle (10pt);
				\fill (0,3.38) circle (10pt);
				\fill (-3,5) circle (10pt);
				
				\node at (0,11) {$7$};
				\node at (0,2) {$16$};
				\node at (-7.25,0) {$18$};
				\node at (7,0) {$9$};
				\node at (-4,5) {$8$};
				\node at (4.25,5) {$17$};
			\end{tikzpicture}\caption{The Pasch switch transforming the triples of $\mathcal{S}_1$ into the triples of $\mathcal{S}_2$.}\label{fig:paschswitch2}
		\end{center}
	\end{figure}

	Therefore, we find the representation of the $\mathrm{STS}(19)$ $\mathcal{S}_2$ shown in \Cref{tab:sts19veblen2}, where the columns are the triples of $\mathcal{S}_2$.
	
	\begin{table}[H]
		\centering
		\begin{tabular}{*{19}{c}}
			\hline
			$\! 0 \!$ & $\! 0 \!$ & $\! 0 \!$ & $\! 0 \!$ & $\! 0 \!$ & $\! 0 \!$ & $\! 0 \!$ & $\! 0 \!$ & $\! 0 \!$ & $\! 1 \!$ & $\! 1 \!$ & $\! 1 \!$ & $\! 1 \!$ & $\! 1 \!$ & $\! 1 \!$ & $\! 1 \!$ & $\! 1 \!$ & $\! 2 \!$ & $\! 2 \!$ \\
			$\! 1 \!$ & $\! 2 \!$ & $\! 3 \!$ & $\! 4 \!$ & $\! 5 \!$ & $\! 6 \!$ & $\! 7 \!$ & $\! 8 \!$ & $\! 9 \!$ & $\! 2 \!$ & $\! 4 \!$ & $\! 5 \!$ & $\! 6 \!$ & $\! 11 \!$ & $\! 13 \!$ & $\! 14 \!$ & $\! 15 \!$ & $\! 4 \!$ & $\! 5 \!$ \\
			$\! 10 \!$ & $\! 11 \!$ & $\! 12 \!$ & $\! 13 \!$ & $\! 14 \!$ & $\! 15 \!$ & $\! 16 \!$ & $\! 17 \!$ & $\! 18 \!$ & $\! 3 \!$ & $\! 7 \!$ & $\! 9 \!$ & $\! 8 \!$ & $\! 12 \!$ & $\! 16 \!$ & $\! 18 \!$ & $\! 17 \!$ & $\! 9 \!$ & $\! 8 \!$ \\
			\hline
			\\ 
			\hline
			$\! 2 \!$ & $\! 2 \!$ & $\! 2 \!$ & $\! 2 \!$ & $\! 2 \!$ & $\! 3 \!$ & $\! 3 \!$ & $\! 3 \!$ & $\! 3 \!$ & $\! 3 \!$ & $\! 3 \!$ & $\! 3 \!$ & $\! 4 \!$ & $\! 4 \!$ & $\! 4 \!$ & $\! 4 \!$ & $\! 4 \!$ & $\! 5 \!$ & $\! 5 \!$ \\
			$\! 6 \!$ & $\! 10 \!$ & $\! 13 \!$ & $\! 14 \!$ & $\! 15 \!$ & $\! 4 \!$ & $\! 5 \!$ & $\! 6 \!$ & $\! 9 \!$ & $\! 10 \!$ & $\! 13 \!$ & $\! 14 \!$ & $\! 5 \!$ & $\! 10 \!$ & $\! 11 \!$ & $\! 12 \!$ & $\! 14 \!$ & $\! 10 \!$ & $\! 11 \!$ \\
			$\! 7 \!$ & $\! 12 \!$ & $\! 18 \!$ & $\! 17 \!$ & $\! 16 \!$ & $\! 8 \!$ & $\! 7 \!$ & $\! 18 \!$ & $\! 15 \!$ & $\! 11 \!$ & $\! 17 \!$ & $\! 16 \!$ & $\! 6 \!$ & $\! 16 \!$ & $\! 18 \!$ & $\! 17 \!$ & $\! 15 \!$ & $\! 18 \!$ & $\! 17 \!$ \\
			\hline
			\\ 
			\hline
			$\! 5 \!$ & $\! 5 \!$ & $\! 6 \!$ & $\! 6 \!$ & $\! 6 \!$ & $\! 6 \!$ & $\! 7 \!$ & $\! 7 \!$ & $\! 7 \!$ & $\! 7 \!$ & $\! 7 \!$ & $\! 8 \!$ & $\! 8 \!$ & $\! 8 \!$ & $\! 8 \!$ & $\! 9 \!$ & $\! 9 \!$ & $\! 12 \!$ & $\! 16 \!$ \\
			$\! 12 \!$ & $\! 13 \!$ & $\! 9 \!$ & $\! 10 \!$ & $\! 11 \!$ & $\! 13 \!$ & $\! 8 \!$ & $\! 9 \!$ & $\! 10 \!$ & $\! 11 \!$ & $\! 12 \!$ & $\! 9 \!$ & $\! 10 \!$ & $\! 11 \!$ & $\! 12 \!$ & $\! 10 \!$ & $\! 11 \!$ & $\! 15 \!$ & $\! 17 \!$ \\
			$\! 16 \!$ & $\! 15 \!$ & $\! 12 \!$ & $\! 17 \!$ & $\! 16 \!$ & $\! 14 \!$ & $\! 18 \!$ & $\! 17 \!$ & $\! 13 \!$ & $\! 15 \!$ & $\! 14 \!$ & $\! 16 \!$ & $\! 15 \!$ & $\! 14 \!$ & $\! 13 \!$ & $\! 14 \!$ & $\! 13 \!$ & $\! 18 \!$ & $\! 18 \!$ \\
			\hline
		\end{tabular}
		\caption{The $\mathrm{STS}(19)$ $\mathcal{S}_1$ with unique Veblen point 0.}\label{tab:sts19veblen2}
	\end{table}

			\section{Steiner triple systems of order $27$}
		By \Cref{eq:numvebpoints}, the number of Veblen points of a Steiner triple system $\mathcal{S}$ of order $27$ is at most $1$. Hence, we can obtain $\mathcal{L_S}$ as a Schreier extension of its center $\mathcal{L_N}$, which is the group of order $2$, by a Steiner loop $\mathcal{L_Q}$ of order $14$ corresponding to one of the two non-isomorphic $\mathrm{STS}(13)$s. The following Theorem determines the number of $\mathrm{STS}(27)$s with (exactly) one Veblen point, investigating both cases for the quotient system $\mathcal{Q}$. 
		
		\begin{theorem}
			There are $1736$ non-isomorphic $\mathrm{STS}(27)$s containing (exactly) one Veblen point. Among these, $1504$ have the non-cyclic $\mathrm{STS}(13)$ as the quotient system, and $232$  have the cyclic $\mathrm{STS}(13)$ as the quotient system.
		\end{theorem}
		
			\begin{proof}
				Let  $\mathcal{S}$ be an $\mathrm{STS}(27)$ with one Veblen point, let $\mathcal{L_N}$ be the center of $\mathcal{L_S}$ and $\mathcal{L_Q}$ the corresponding quotient loop. Since $\mathcal{L_N}$ has order $2$ and $\mathcal{Q}$ has $26$ triples, the total number of all possible factor systems in $\mathrm{Ext}_S(\mathcal{L_N},\mathcal{L_Q})$ is $2^{26}$.
				
				Now, we apply the algorithms described in \Cref{sec:descriptionalgorithms}. The set of co-boundaries $\mathrm{B}^2$ has order $2^{13}$, so the number of non-equivalent extensions is $2^{13}$. We note now that, since $\mathrm{Aut}(\mathcal{L_N})$ is trivial, the action of the group $\mathrm{Aut}(\mathcal{L_N})\times\mathrm{Aut}(\mathcal{L_Q})$ can be reduced to the action of just $\mathrm{Aut}(\mathcal{L_Q})$, which can be either $S_3$ or $F_{39}$.
				
				If $\mathcal{Q}$ is the non-cyclic $\mathrm{STS}(13)$, then $\mathrm{Aut}(\mathcal{L_Q})$ is the symmetric group $S_3$. Under its action, the set of non-equivalent extensions is divided into $1504$ orbits, each of which represents one isomorphism class of extensions of $\mathcal{L_N}$ by $\mathcal{L_Q}$. Since $\mathcal{L_N}$ is the center of $\mathcal{L_S}$, these are exactly the isomorphism classes of $\mathrm{STS}(27)$s with one Veblen point and quotient system the non-cyclic $\mathrm{STS}(13)$.
		
				Let now $\mathcal{Q}$ be the cyclic $\mathrm{STS}(13)$. In this case $\mathrm{Aut}(\mathcal{L_Q})$ is the subgroup of $\mathrm{AGL}(1,13)$ of all the affine transformations of the type $x\mapsto ax+b$ with $x,b\in\mathbb{Z}/13\mathbb{Z}$ and $a\in \{1,3,9\}$, which is, up to isomorphism, the unique Frobenius group $F_{39}$ of order 39. Under its action, the set of non-equivalent extensions is divided into $232$ orbits, which, for the same reason as the previous case, are exactly the isomorphism classes of $\mathrm{STS}(27)$s with one Veblen point and quotient system the cyclic $\mathrm{STS}(13)$.	
		\end{proof}

		\section{Steiner triple systems of order $31$}
		Lastly, by \Cref{eq:numvebpoints} and \Cref{theoindex4}, we deduce that the number of Veblen points of an $\mathrm{STS}(31)$ can either be one, three, or thirty-one. In the last case, it is the projective geometry $\mathrm{PG}(4,2)$, thus we focus to the cases of $\mathrm{STS}(31)$s with one or three Veblen points.
		
		If a Steiner triple system of order $31$ contains one Veblen point, then the corresponding quotient system $\mathcal{Q}$ is an $\mathrm{STS}(15)$. Since there are $80$ possibilities for such a quotient system (see  \cite{Handbook}), we focus our attention on six cases that in our opinion are the most interesting, namely: 
		\begin{enumerate}
			\item $\mathrm{STS}(15)$ $\# 1$, that is, $\mathrm{PG}(3,2)$;
			\item $\mathrm{STS}(15)$ $\# 2$, which is the only other one with  a Veblen point itself;
			\item $\mathrm{STS}(15)$ $\# 3$, which is the one with the largest number of Pasch configurations without containing Veblen points;
			\item $\mathrm{STS}(15)$ $\# 7$, which is the one with the second-largest automorphism group after $\mathrm{PG}(3,2)$; 
			\item $\mathrm{STS}(15)$ $\# 61$, which, among the ones containing a Fano plane, is the one with the least number of Pasch configurations;
			\item $\mathrm{STS}(15)$ $\# 80$, which is the only one containing no Pasch configurations.
		\end{enumerate}
		 
		 We have the following partial result. 
		
			\begin{theorem}
				The number of non-isomorphic $\mathrm{STS}(31)$s with exactly one Veblen point and quotient system $\mathcal{Q}$ of order $15$ is given by \Cref{tab:sts31with1}.
				\begin{table}[H]
					\centering
					\begin{tabular}{|c|c|}
						\hline
						$\mathcal{Q}$ & count \\
						\hline
						$\mathrm{PG}(3,2)$ & $278$ \\
						$\mathrm{STS}(15) \# 2$ & $48072$ \\
						$\mathrm{STS}(15) \# 3$ & $47744$ \\
						$\mathrm{STS}(15) \# 7$ & $16520$ \\
						$\mathrm{STS}(15) \# 61$ & $99952$ \\
						$\mathrm{STS}(15) \# 80$ & $17888$ \\
						\hline
					\end{tabular}\caption{The number of $\mathrm{STS}(31)$s with one Veblen point and corresponding quotient system $\mathcal{Q}$.}\label{tab:sts31with1}
				\end{table}
			\end{theorem}
		
		\begin{proof}
				Let $\mathcal{S}$ be an $\mathrm{STS}(31)$ with (at least) one Veblen point. We can obtain the Steiner loop $\mathcal{L_S}$ as a Schreier extension of its central subloop $\mathcal{L_N}$ of order $2$ containing the given Veblen point, by a Steiner loop $\mathcal{L_Q}$ of order $16$ corresponding to an $\mathrm{STS}(15)$. Since $\mathcal{L_N}$ has order $2$ and $\mathcal{Q}$ has $35$ triples, in any case, the total number of all possible factor systems in $\mathrm{Ext}_S(\mathcal{L_N},\mathcal{L_Q})$ is $2^{35}$.
				
				Now, we apply the algorithms described in \Cref{sec:descriptionalgorithms}. The order of the set of co-boundaries $\mathrm{B}^2$ is described, in all of the $6$ cases that we want to analyze, by \Cref{tab:numcoboundQ}.
				\begin{table}[H]
					\centering
					\begin{tabular}{|c|c|}
						\hline
						$\mathcal{Q}$ & $|\mathrm{B}^2|$  \\
						\hline 
						$\mathrm{PG}(3,2)$ & $2^{11}$  \\
						$\mathrm{STS}(15)\ \# 2$ & $2^{12}$  \\
						$\mathrm{STS}(15)\ \# 3$ & $2^{13}$  \\
						$\mathrm{STS}(15)\ \# 7$ & $2^{13}$ \\
						$\mathrm{STS}(15)\ \# 61$ & $2^{14}$  \\
						$\mathrm{STS}(15)\ \# 80$ & $2^{15}$  \\
						\hline
					\end{tabular}\caption{The number of coboundaries for the corresponding $\mathcal{Q}$.}\label{tab:numcoboundQ}
				\end{table}
				Consequently, we can count the non-equivalent extensions in each case, as described by  \Cref{tab:noneqQ}.
				\begin{table}[H]
					\centering
					\begin{tabular}{|c|c|}
						\hline
						$\mathcal{Q}$ & non-equivalent extensions  \\
						\hline 
						$\mathrm{PG}(3,2)$ & $2^{24}$  \\
						$\mathrm{STS}(15)\ \# 2$ & $2^{23}$  \\
						$\mathrm{STS}(15)\ \# 3$ & $2^{22}$  \\
						$\mathrm{STS}(15)\ \# 7$ & $2^{22}$ \\
						$\mathrm{STS}(15)\ \# 61$ & $2^{21}$  \\
						$\mathrm{STS}(15)\ \# 80$ & $2^{20}$  \\
						\hline
					\end{tabular}\caption{Number of non equivalent factor systems for the corresponding $\mathcal{Q}$.}\label{tab:noneqQ}
				\end{table}

				We note now that, since $\mathrm{Aut}(\mathcal{L_N})$ is trivial, the action of the group $\mathrm{Aut}(\mathcal{L_N})\times\mathrm{Aut}(\mathcal{L_Q})$ can be reduced to the action of just $\mathrm{Aut}(\mathcal{L_Q})$. The order of this group in every case is given in \Cref{tab:autQ}.
				
				\begin{table}[H]
					\centering
					\begin{tabular}{|c|c|}
						\hline
						$\mathcal{Q}$ & $|\mathrm{Aut}(\mathcal{Q})|$\\
						\hline
						$\mathrm{PG}(3,2) $& $20160$ \\
						$\mathrm{STS}(15) \# 2 $& $192$\\
						$\mathrm{STS}(15) \# $3 & $96$ \\
						$\mathrm{STS}(15) \# 7$ & $288$ \\
						$\mathrm{STS}(15) \# 61$ & $21$ \\
						$\mathrm{STS}(15) \# $80 & $60$ \\
						\hline
					\end{tabular}\caption{Order of the automorphism group of $\mathcal{Q}$.}\label{tab:autQ}
				\end{table}	
				\noindent Under this action, the equivalence classes of factor systems are divided into $1240$ orbits for the $\mathrm{STS}(15)$ $\# 1$, $48080$ for $\# 2$, $47744$ for  $\# 3$, $16520$ for $\# 7$, $99952$ for  $\# 61$, and $17888$ for  $\# 80$. 
				
				Since  $\mathrm{STS}(15)$ $\# 3$, $7$, $61$, $80$ have no Veblen points, in these cases $\mathcal{L_N}$ is the whole center of $\mathcal{L_S}$. This implies that the orbits coincide with the isomorphism classes of the corresponding $\mathrm{STS}(31)$s. 
				
				For the $\mathrm{STS}(15)$s $\# 1$ and $2$, the situation is different, due to the fact that they have Veblen points themselves. This means that $\mathcal{L_N}$ could be (in general) just a proper subgroup of the center of $\mathcal{L_S}$, that is, $\mathcal{S}$ could have more than one Veblen point. For this reason we need a further reduction, in order to cut out all the factor systems which produce an $\mathrm{STS}(31)$ with more than $1$ Veblen point. As seen in \Cref{rmk:furtherveblen}, we need to count just the factor systems $f$ for which does not exist any point $P\in\mathcal{Q}$ that fulfills the following condition
				\begin{equation*}
					f(P,Q)+f(PQ,R)=f(Q,R)+f(P,QR),
				\end{equation*}
				for every $Q,R\in\mathcal{Q}$. 
				Since the $\mathrm{STS}(15)$ $\# 1$ is projective, this study needs to be done considering all the points of the Steiner triple system. For the $\mathrm{STS}(15)$ $\# 2$ we can consider just the point labeled with $0$ (see \cite[Table 1.28, p. 30]{Handbook}) since, in this case, it is the only Veblen point. After the computation, we found out that there are $278$ non-isomorphic $\mathrm{STS}(31)$s with precisely one Veblen point and quotient system the $\mathrm{STS}(15)$ $\# 1$, and $48072$ non-isomorphic $\mathrm{STS}(31)$s with precisely one Veblen point and quotient system the $\mathrm{STS}(15)$ $\# 2$.
			\end{proof}

	As we observed in this last case of $\mathrm{STS}(31)$s with precisely one Veblen point, the problem of constructing all Steiner triple systems of given size and number of Veblen points using Schreier extensions becomes progressively challenging with size growth, in particular when the ratio of the order and the number of Veblen points becomes larger. This is due to the fact that, if the number of Veblen points is relatively small, the number of non-isomorphic cases for the quotient system $\mathcal{Q}$ increases sensibly. Indeed, the problem of enumerating $\mathrm{STS}(31)$s with three Veblen points is much easier to deal with, as shown in the next result.

	\begin{theorem}\label{th313}
	There are only $2$ non-isomorphic $\mathrm{STS}(31)$s with exactly three Veblen points.
\end{theorem}

\begin{proof}
	Let $\mathcal{S}$ be an $\mathrm{STS}(31)$ with (at least) three Veblen points. We can obtain the Steiner loop $\mathcal{L_S}$ as a Schreier extension of its central subloop $\mathcal{L_N}$ corresponding to the three given Veblen points, which is the elementary abelian  $2$-group of order $4$, by the unique Steiner loop $\mathcal{L_Q}$ of order $8$ corresponding to the $\mathrm{STS}(7)$. Since $\mathcal{L_N}$ has order $4$ and $\mathcal{Q}$ has $7$ triples, the total number of all possible factor systems in $\mathrm{Ext}_S(\mathcal{L_N},\mathcal{L_Q})$ is $4^{7}=16383$.	The set of co-boundaries $\mathrm{B}^2$ has order $2^8=256$. 
	Consequently, the number of non-equivalent extensions is  $2^6=64$.
	The group $\mathrm{Aut}(\mathcal{L_N})\times\mathrm{Aut}(\mathcal{L_Q})$ is $ \mathrm{PGL}(2,2)\times \mathrm{PGL}(3,2)$, which has order $1008$.  Under its action, the $64$ equivalence classes of factor systems are divided into $3$ orbits. 
	
	Since $\mathcal{Q}$ has Veblen points itself, $\mathcal{L_N}$ could be, in general, a proper subgroup of the center of $\mathcal{L_S}$. Hence $\mathcal{S}$ could have more than three Veblen points, and that would be the case where $\mathcal{S}$ is projective. Therefore, these orbits do not necessarily coincide with the isomorphism classes of $\mathrm{STS}(19)$s with three Veblen points. We need to perform an analysis to see if some of these $3$ resulting factor systems produce a projective geometry.
	
	In order to do this, let us to describe explicitly the three obtained  factor systems. We see $\mathcal{L_N}$ as $\mathrm{GF}(2)^2$ and $\mathcal{L_Q}$ as $\mathrm{GF}(2)^3$, and we give the Fano plane $\mathcal{Q}$ the following coordinates, as in \Cref{fig:fanoforsts31}.
	\begin{equation*}
		P_1=[0,0,1], \quad P_2=[0,1,0], \quad P_3=[0,1,1], \quad P_4=[1,0,0], \quad 	P_5=[1,0,1], \quad P_6=[1,1,0], \quad P_7=[1,1,1].
	\end{equation*}
	
		\begin{figure}[H]
		\begin{center}
			\begin{tikzpicture}[scale=0.35]
				\draw[line width=1] (-6,0) -- (6,0) -- (0,10) -- cycle;
				\draw[line width=1] (-6,0) -- (3,5);
				\draw[line width=1] (6,0) -- (-3,5);
				\draw[line width=1] (0,10) -- (0,0);
				
				\fill (-6,0) circle (10pt);
				\fill (6,0) circle (10pt);
				\fill (0,10) circle (10pt);
				\fill (0,0) circle (10pt);
				\fill (3,5) circle (10pt);
				\fill (0,3.38) circle (10pt);
				\fill (-3,5) circle (10pt);
				
				\draw[line width=1] (0,3.38) circle [radius=3.38];
				
				\node at (0,11) {$P_4$};
				\node at (0,-1) {$P_3$};
				\node at (2,3.38) {$P_7$};
				\node at (-7.5,0) {$P_1$};
				\node at (7.5,0) {$P_2$};
				\node at (-4,5) {$P_5$};
				\node at (4,5) {$P_6$};
			\end{tikzpicture}\caption{The Fano plane.}\label{fig:fanoforsts31}
		\end{center}
	\end{figure}
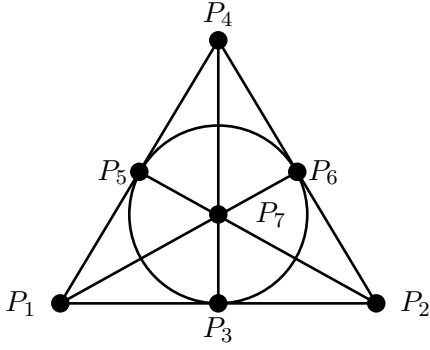
	
	The three orbits that we obtained are represented, respectively, by the following three factor systems:
	\begin{align*}
		&f_0, \quad \text{the null function, i.e., the function mapping every pair into $\Omega$;}\\
		&f_1, \quad \text{defined by } f_1(P_3,P_5)=f_1(P_3,P_6)=f_1(P_5,P_6)=(0,1), \quad \text{and $\Omega$ elsewhere;}\\
		&f_2, \quad \text{defined by } f_2(P_3,P_4)=f_2(P_3,P_7)=f_2(P_4,P_7)=(0,1), \ f_2(P_3,P_5)=f_2(P_3,P_6)=f_2(P_5,P_6)=(1,0)\\ &\quad \quad \ \text{and $\Omega$ elsewhere.}
	\end{align*}
	The Steiner loop defined by the null factor system $f_0$ is a group, hence the corresponding $\mathrm{STS}(31)$ is $\mathrm{PG}(4,2)$, which has $31$ Veblen points.
	
	\Cref{fig:factorsystemsf1f2} provides a visual representation of the non-trivial factor systems $f_1$ and $f_2$, where the triples of $\mathcal{Q}$ in which $f_i$ is equal to $(0,1)$ are drawn in red, and the ones in which $f_i$ is equal to $(1,0)$ are drawn in blue, $i=1,2$.
	
	\begin{figure}[H]
		\begin{center}
			\begin{tikzpicture}[scale=0.35]
				
				\draw[line width=0.5pt,color=gray]  (-6,0) -- (6,0) -- (0,10) -- cycle;
				\draw[line width=0.5pt,color=gray]  (-6,0) -- (3,5);
				\draw[line width=0.5pt,color=gray]  (6,0) -- (-3,5);
				\draw[line width=0.5pt,color=gray]  (0,10) -- (0,0);
				
				\draw[line width=1.5pt,color=red] (0,3.38) circle [radius=3.38];
				
				\fill (-6,0) circle (10pt);
				\fill (6,0) circle (10pt);
				\fill (0,10) circle (10pt);
				\fill (0,0) circle (10pt);
				\fill (3,5) circle (10pt);
				\fill (0,3.38) circle (10pt);
				\fill (-3,5) circle (10pt); 
				
				\node at (0,11) {$P_4$};
				\node at (0,-1) {$P_3$};
				\node at (2,3.38) {$P_7$};
				\node at (-7.5,0) {$P_1$};
				\node at (7.5,0) {$P_2$};
				\node at (-4,5) {$P_5$};
				\node at (4,5) {$P_6$};
			\end{tikzpicture}
			\hspace{1cm}
			\begin{tikzpicture}[scale=0.35]
				
				\draw[line width=0.5pt,color=gray]  (-6,0) -- (6,0) -- (0,10) -- cycle;
				\draw[line width=0.5pt,color=gray]  (-6,0) -- (3,5);
				\draw[line width=0.5pt,color=gray]  (6,0) -- (-3,5);
				\draw[line width=1.5pt,color=red]   (0,10) -- (0,0);
				
				\draw[line width=1.5pt,color=blue] (0,3.38) circle [radius=3.38];
				
				\fill (-6,0) circle (10pt);
				\fill (6,0) circle (10pt);
				\fill (0,10) circle (10pt);
				\fill (0,0) circle (10pt);
				\fill (3,5) circle (10pt);
				\fill (0,3.38) circle (10pt);
				\fill (-3,5) circle (10pt); 
				
				\node at (0,11) {$P_4$};
				\node at (0,-1) {$P_3$};
				\node at (2,3.38) {$P_7$};
				\node at (-7.5,0) {$P_1$};
				\node at (7.5,0) {$P_2$};
				\node at (-4,5) {$P_5$};
				\node at (4,5) {$P_6$};
			\end{tikzpicture}\caption{The factor systems $f_1$ (left) and $f_2$ (right).}\label{fig:factorsystemsf1f2}
		\end{center}
	\end{figure}
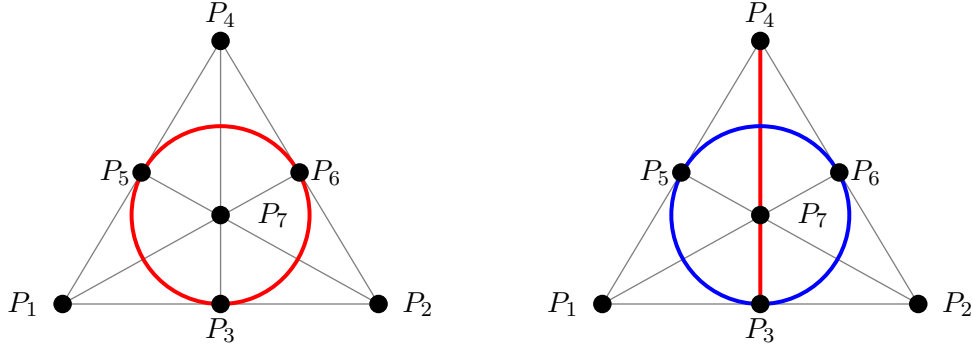
	
	For $i=1,2$, it is easy to check by hand that for any $P\in\mathcal{Q}$, no element of the form $(P,x)$ can be a Veblen point of $\mathcal{S}_i$ since the following condition is never satisfied:
	\begin{equation*}
		f_i(P,Q)+f_i(PQ,R)=f_i(Q,R)+f_i(P,QR) \quad \text{for every $Q,R\in\mathcal{Q}$}.
	\end{equation*}
	Hence, $\mathcal{L_N}$ is in both cases the whole center of $\mathcal{L}_{\mathcal{S}_i}$, $i=1,2$.  In conclusion, there are only $2$ non-isomorphic $\mathrm{STS}(31)$s with precisely three Veblen points. 
\end{proof}

	The two non-trivial factor systems defined in the proof of \Cref{th313} give a compact representation of all the $155$ triples of each of the two $\mathrm{STS}(31)$s with exactly $3$ Veblen points.

	\nocite{*}
	\printbibliography[title={References}]

	\newpage 
	
	\appendix
	\section{Appendix}\label{appendix}


In this section, we present a description of the algorithms used for our results. In particular, we developed our code in Python by using
the library Sagemath 10.2 \footnote{\url{www.sagemath.org}}. The source code is available at \url{https://github.com/giuseppefilippone/steiner-triple-systems-vb-19-27-31}. Moreover, we run it on a machine with the
following hardware:
\begin{itemize}
	\item HDD: 4.8 TB SAS ($4 \times 1.2$ TB);
	\item RAM: 128 GB DDR4 DIMM 2933 MHz ($4 \times 32$ GB);
	\item CPU: 2 Intel Xeon Gold 5218, 16-core.
\end{itemize}

\subsection{Description of the algorithms}\label{sec:descriptionalgorithms}

In order to make the results presented in this work more clear, we describe the algorithms used for classifying the Steiner triple systems of order $19$, $27$ and $31$ with Veblen points. 

In general, we begin with an associative Steiner loop $\mathcal{L_N}$, which, being an elementary abelian $2$-group, we represent by $\mathrm{GF}(2)^t$. In our specific cases, $t$ will be either $1$ or $2$. Additionally, we have a Steiner loop $\mathcal{L_Q}$ whose the operation is defined by the triples of $\mathcal{Q}$ arranged as columns of a table. Fixed an order for the $b$ triples of $\mathcal{Q}$, we define the ordered set of \emph{fundamental pairs}, denoted with $\mathcal{Q}_2$, as the family of $2$-subsets obtained by removing one element from each triple. 

Since any factor system $f\in\mathrm{Ext_S}(\mathcal{L_N},\mathcal{L_Q})$ is defined by the value it takes in every fundamental pair, in order to improve the efficiency of our implementation we can represent $f$ with an integer in the following way. Let $n_i\in\mathrm{GF}(2)^t$ be the image of the fundamental pair $c_i$ under $f$, $i=1,\dots,b$. Then $f$ can be identified uniquely with the vector $(n_1,n_2,\dots,n_b)\in(\mathrm{GF}(2)^{t})^b$. Of course $f$ can be compressed into a binary vector of length ${tb}$, which is the binary representation of an integer. 

On the other hand, every vector $v$ of $\mathrm{GF}(2)^{tb}$ defines uniquely a factor system $f$, precisely the one mapping the $i$-th fundamental pair into the $i$-th sub-vector of $v$ of length $t$, $i=1,\dots,b$. In this way, we can uniquely describe every factor system as an integer $0\leq k \leq 2^{tb}-1$.

\begin{example}
	For instance, let $\mathcal{L_N}$ be $\mathrm{GF}(2)^2$ and $\mathcal{Q}$ be the $\mathrm{STS}(9)$ with points $\{1,\dots, 9\}$ and triples given by the columns of \Cref{tab:sts9}. 
	\begin{table}[H]
		\centering
		\begin{tabular}{*{12}{c}} 
			1 & 1 & 1 & 1 & 2 & 2 & 2 & 3 & 3 & 3 & 4 & 7\\
			2 & 4 & 5 & 6 & 4 & 5 & 6 & 4 & 5 & 6 & 5 & 8\\
			3 & 7 & 9 & 8 & 9 & 8 & 7 & 5 & 7 & 9 & 6 & 9\\
		\end{tabular}
		\caption{Triples of the $ \mathrm{STS}(9) $}
		\label{tab:sts9}
	\end{table}	
	\noindent Let the fundamental pairs of $\mathcal{Q}$ be given in lexicographic order as follows:
	\begin{align*}
		\{1,2\}, \ \{1,4\}, \ \{1,5\}, \ \{1,6\}, \ \{2,4\},\ \{2,5\},\\
		\{2,6\}, \  \{3,4\}, \ \{3,5\}, \ \{3,6\}, \ \{4,5\}, \ \{7,8\}.
	\end{align*}
	Let $f$ be the factor system defined by the following conditions:
	\begin{align*}
		&f(1,2)=(0,0), \ f(1,4)=(0,0), \ f(1,5)=(0,0), \ f(1,6)=(0,0), \\ 
		&f(2,4)=(0,1),\ f(2,5)=(1,1), \ f(2,6)=(0,0), \  f(3,4)=(1,1), \\
		&f(3,5)=(1,0), \ f(3,6)=(0,1), \ f(4,5)=(0,0), \ f(7,8)=(0,0).
	\end{align*}
	We can represent $f$ as the vector
	{\small 
		\begin{equation*}
			\left( (0,0),(0,0),(0,0),(0,0),(0,1),(1,1),(0,0),(1,1),(1,0),(0,1),(0,0),(0,0)\right)
	\end{equation*}}
	of $(\mathrm{GF}(2)^2)^6$ which can be compacted into the binary vector $$(0,0,0,0,0,0,0,0,0,1,1,1,0,0,1,1,1,0,0,1,0,0,0,0)$$ 
	representing the integer $29584$.
\end{example}

In the same way, after computing the set $\mathrm{B}^2(\mathcal{L_Q},\mathcal{L_N})=\{\delta^1\varphi \mid \varphi\colon \mathcal{L_Q}\to\mathcal{L_N}, \ \varphi(\bar \Omega)=\Omega\}$, we can represent every coboundary as one of the integers $k$, $0\leq k \leq 2^{tb}-1$. 
In order to find the cosets in $\mathrm{Ext}_S(\mathcal{L_N},\mathcal{L_Q})/\mathrm{B}^2(\mathcal{L_Q},\mathcal{L_N})$, we compute the sum $f+\delta^1\varphi$ as the operation XOR between binary vectors.  Consequently, we can find the set of non-equivalent extensions by taking one representative in each coset.

Now, to find the set of non-isomorphic extensions, we compute the full automorphism groups $\mathrm{Aut}(\mathcal{L_N})$ and $\mathrm{Aut}(\mathcal{L_Q})$. In our cases, the former will be either the trivial group or $\mathrm{GL}(2,2)$. For the latter, since the quotient system $\mathcal{Q}$ in our cases have order either $7$, $9$, $13$ or $15$, we have different possibilities. 

If $\mathcal{Q}$ is the $\mathrm{STS}(7)$, the $\mathrm{STS}(9)$ or the $\mathrm{STS}(15)$ $\# 1$, then the group $\mathrm{Aut}(\mathcal{L_Q})$ is $\mathrm{GL}(3,2)$, $\mathrm{Aff}(2,3)=\mathrm{GF}(3)^2\rtimes \mathrm{GL}(2,3)$, or $\mathrm{GL}(4,2)$, respectively.

If $\mathcal{Q}$ is an $\mathrm{STS}(13)$, then there are the two possibilities.
One of the two non-isomorphic $\mathrm{STS}(13)$s is usually referred to as the \emph{non-cyclic} $\mathrm{STS}(13)$, and in this case $\mathrm{Aut}(\mathcal{L_Q})$ is the  the symmetric group $S_3$. The other one is called the \emph{cyclic} $\mathrm{STS}(13)$, because it has an automorphism of order $13$, and the full group $\mathrm{Aut}(\mathcal{L_Q})$ is the unique non-abelian group of order $39$, that is, the Frobenius group $F_{39}$. 
Let us denote, in both cases, the set of points of the $\mathrm{STS}(13)$ with $\{0,1,\dots,12\}$.  In the first case, with the triples given by \Cref{tab:sts13-1}, $\mathrm{Aut}(\mathcal{L_Q})$ is computed as the group generated by the function $x\mapsto 3x\pmod{13}$ and the permutation $(6 \; 8)(2 \; 11)(3 \; 9)(4 \; 12)(5 \; 7)$. In the second case, with the triples given by \Cref{tab:sts13-2}, $\mathrm{Aut}(\mathcal{L_Q})$ is computed as the group of functions $x\mapsto ax+b\pmod{13}$, where $a\in\{1,3,9\}$ and $b\in\{0,\dots,12\}$ (see, e.g., \cite{Pavone13}).

If $\mathcal{Q}$ is a non-projective $\mathrm{STS}(15)$, we compute the group $\mathrm{Aut}(\mathcal{L_Q})$ as the group of permutations of the $15$ points which induce a permutation of the $35$ triples as well. 

Let us now consider a factor systems $f$ represented as a vector $(n_1,n_2,\dots,n_b)\in(\mathrm{GF}(2)^t)^b$ and two automorphisms $\alpha\in\mathrm{Aut}(\mathcal{L_N})$, $\beta\in\mathrm{Aut}(\mathcal{L_N})$. The factor systems $\alpha f$ and $f\beta^{-1}$ are represented, respectively, by the vectors
\begin{equation*}
	\left(\alpha(n_1),\alpha(n_2),\dots,\alpha(n_b)\right) \quad \text{and} \quad \left(n_{\sigma(1)},n_{\sigma(2)},\dots,n_{\sigma(b)}\right).
\end{equation*}
where $\sigma\in S_b$ is the permutation induced by the automorphism $\beta$ on the set of indexes of the the fundamental pairs. 
After computing the action of both $\mathcal{L_N}$ and $\mathcal{L_Q}$ on the set of non-equivalent extensions, we take one representative for each orbit.


The main bottleneck of our implementation is given by the computation of all the factor systems in $\mathrm{Ext_S}(\mathcal{L_N},\mathcal{L_Q})$. More specifically, in the case of Steiner triple system of order $31$, 
the order of this group is $2^{35}$, and since we represent each factor system as an integer up to $35$ bits, the total dimension of $\mathrm{Ext_S}(\mathcal{L_N},\mathcal{L_Q})$
is $ \sum_{i = 0}^{b = 35} 2^i = \frac{2^{35} \cdot (2^{35} + 1)}{2} $ bits (approximately, 64 PB). Hence, for this reason,
we implemented a batch approach to compute all the non-equivalent extensions. In particular, we load blocks of factor systems
(see function \textbf{coset\_blocks} in the source code) and we check for new non-equivalent factor systems in each block, 
comparing each factor system with those already identified so far. 

This approach is also parallelizable, but it requires at least a thousand CPU cores to perform the computation. 
Additionally, computation times could be much longer than our version due to internal communication between 
the various cores (in the case of parallelization with dedicated memory), 
or due to non-blocking access to shared memory, if any. 

However, it will be interesting to check if parallelization could reduce the actual overall computation time, which, for instance, was of approximately three days for the case of the $\mathrm{STS}(31)$s with $\mathcal{Q} = \mathrm{PG}(3,2)$.

Here, we write the tables representing the Steiner triple systems mentioned in the previous sections and used in the algorithms.

\begin{table}[H]
	\centering
	\begin{tabular}{*{12}{c}} 
		1 & 1 & 1 & 1 & 2 & 2 & 2 & 3 & 3 & 4 & 4 & 6\\
		2 & 3 & 4 & 5 & 3 & 5 & 6 & 5 & 7 & 5 & 8 & 7\\
		8 & 6 & 7 & 9 & 4 & 7 & 9 & 8 & 9 & 6 & 9 & 8\\
	\end{tabular}
	\caption{$ \mathrm{STS}(9) $}
	\label{tab:sts9A}
\end{table}

\begin{table}[H]
	\centering
	\resizebox{0.9\textwidth}{!}{
	\begin{tabular}{*{26}{c}} 
		1 &  2 &  3 &  4 & 5 & 6 & 7 &  8 &  9 &  6 & 11 & 12 & 0 &  2 & 3 & 4 & 5 &  6 &  7 &  8 &  9 & 10 & 11 & 12 & 0 & 1\\
		3 &  4 &  5 & 10 & 7 & 8 & 9 & 10 & 11 & 12 &  0 &  1 & 2 &  5 & 6 & 7 & 8 &  9 & 10 & 11 & 12 &  0 &  1 &  2 & 3 & 4\\
		9 &  6 & 11 & 12 & 0 & 1 & 2 &  3 &  4 &  5 &  6 &  7 & 8 & 10 & 7 & 8 & 9 & 10 & 11 & 12 &  0 &  1 &  2 &  3 & 4 & 5\\
	\end{tabular}}
	\caption{$ \mathrm{STS}(13) $ $\# 1$ (non-cyclic)}
	\label{tab:sts13-1}
\end{table}

\begin{table}[H]
	\centering
	\resizebox{0.9\textwidth}{!}{
	\begin{tabular}{*{26}{c}} 
		1 &  2 &  3 &  4 & 5 & 6 & 7 &  8 &  9 & 10 & 11 & 12 & 0 & 2 & 3 & 4 & 5 &  6 &  7 &  8 &  9 & 10 & 11 & 12 & 0 & 1 \\
		3 &  4 &  5 &  6 & 7 & 8 & 9 & 10 & 11 & 12 &  0 &  1 & 2 & 5 & 6 & 7 & 8 &  9 & 10 & 11 & 12 &  0 &  1 &  2 & 3 & 4 \\
		9 & 10 & 11 & 12 & 0 & 1 & 2 &  3 &  4 &  5 &  6 &  7 & 8 & 6 & 7 & 8 & 9 & 10 & 11 & 12 &  0 &  1 &  2 &  3 & 4 & 5 \\
	\end{tabular}}
	\caption{$ \mathrm{STS}(13) $ $\# 2$ (cyclic)}
	\label{tab:sts13-2}
\end{table}

\begin{table}[H]
	\centering
	\resizebox{0.9\textwidth}{!}{
	\begin{tabular}{*{35}{c}} 
		0 & 0 & 0 & 0 & 0 & 0 & 0 & 1 & 1 & 1 & 1 & 1 & 1 & 2 & 2 & 2 & 2 & 2 & 2 & 3 & 3 & 3 & 3 & 4 & 4 & 4 & 4 & 5 & 5 & 5 & 5 & 6 & 6 & 6 & 6 \\
		1 & 3 & 5 & 7 & 9 & b & d & 3 & 4 & 7 & 8 & b & c & 3 & 4 & 7 & 8 & b & c & 7 & 8 & 9 & a & 7 & 8 & 9 & a & 7 & 8 & 9 & a & 7 & 8 & 9 & a \\
		2 & 4 & 6 & 8 & a & c & e & 5 & 6 & 9 & a & d & e & 6 & 5 & a & 9 & e & d & b & c & d & e & c & b & e & d & e & d & c & b & d & e & b & c \\
	\end{tabular}}
	\caption{$ \mathrm{STS}(15) $ $\# 2$}
	\label{tab:sts15-2}
\end{table}

\begin{table}[H]
	\centering
	\resizebox{0.9\textwidth}{!}{
	\begin{tabular}{*{35}{c}} 
		0 & 0 & 0 & 0 & 0 & 0 & 0 & 1 & 1 & 1 & 1 & 1 & 1 & 2 & 2 & 2 & 2 & 2 & 2 & 3 & 3 & 3 & 3 & 4 & 4 & 4 & 4 & 5 & 5 & 5 & 5 & 6 & 6 & 6 & 6\\
		1 & 3 & 5 & 7 & 9 & b & d & 3 & 4 & 7 & 8 & b & c & 3 & 4 & 7 & 8 & b & c & 7 & 8 & 9 & a & 7 & 8 & 9 & a & 7 & 8 & 9 & a & 7 & 8 & 9 & a\\
		2 & 4 & 6 & 8 & a & c & e & 5 & 6 & 9 & a & d & e & 6 & 5 & a & 9 & e & d & b & c & d & e & d & e & b & c & e & d & c & b & c & b & e & d\\
	\end{tabular}}
	\caption{$ \mathrm{STS}(15) $ $\# 3$}
	\label{tab:sts15-3}
\end{table}

\begin{table}[H]
	\centering
	\resizebox{0.9\textwidth}{!}{
	\begin{tabular}{*{35}{c}} 
		0 & 0 & 0 & 0 & 0 & 0 & 0 & 1 & 1 & 1 & 1 & 1 & 1 & 2 & 2 & 2 & 2 & 2 & 2 & 3 & 3 & 3 & 3 & 4 & 4 & 4 & 4 & 5 & 5 & 5 & 5 & 6 & 6 & 6 & 6\\
		1 & 3 & 5 & 7 & 9 & b & d & 3 & 4 & 7 & 8 & b & c & 3 & 4 & 7 & 8 & b & c & 7 & 8 & 9 & a & 7 & 8 & 9 & a & 7 & 8 & 9 & a & 7 & 8 & 9 & a\\
		2 & 4 & 6 & 8 & a & c & e & 5 & 6 & 9 & a & d & e & 6 & 5 & a & 9 & e & d & b & d & e & c & e & c & b & d & c & e & d & b & d & b & c & e\\
	\end{tabular}}
	\caption{$ \mathrm{STS}(15) $ $\# 7$}
	\label{tab:sts15-7}
\end{table}

\begin{table}[H]
	\centering
	\resizebox{0.9\textwidth}{!}{
	\begin{tabular}{*{35}{c}} 
		0 & 0 & 0 & 0 & 0 & 0 & 0 & 1 & 1 & 1 & 1 & 1 & 1 & 2 & 2 & 2 & 2 & 2 & 2 & 3 & 3 & 3 & 3 & 4 & 4 & 4 & 4 & 5 & 5 & 5 & 5 & 6 & 6 & 6 & 6\\
		1 & 3 & 5 & 7 & 9 & b & d & 3 & 4 & 7 & 8 & a & c & 3 & 4 & 7 & 8 & 9 & b & 7 & 8 & 9 & c & 7 & 8 & 9 & a & 7 & 8 & a & b & 7 & 8 & 9 & a\\
		2 & 4 & 6 & 8 & a & c & e & 5 & 6 & 9 & b & d & e & 6 & 5 & a & e & c & d & b & a & e & d & e & c & d & b & d & 9 & c & e & c & d & b & e\\
	\end{tabular}}
	\caption{$ \mathrm{STS}(15) $ $\# 61$}
	\label{tab:sts15-61}
\end{table}

\begin{table}[H]
	\centering
	\resizebox{0.9\textwidth}{!}{
	\begin{tabular}{*{35}{c}} 
		0 & 0 & 0 & 0 & 0 & 0 & 0 & 1 & 1 & 1 & 1 & 1 & 1 & 2 & 2 & 2 & 2 & 2 & 2 & 3 & 3 & 3 & 3 & 4 & 4 & 4 & 4 & 5 & 5 & 5 & 6 & 6 & 6 & 7 & 8\\
		1 & 3 & 5 & 7 & 9 & b & d & 3 & 4 & 6 & 9 & a & c & 3 & 4 & 5 & 7 & 8 & b & 6 & 7 & 8 & a & 5 & 8 & a & b & 7 & 8 & 9 & 7 & 9 & c & 9 & a\\
		2 & 4 & 6 & 8 & a & c & e & 5 & 7 & 8 & b & d & e & 9 & 6 & a & e & c & d & b & c & d & e & d & 9 & c & e & b & e & c & a & e & d & d & b\\
	\end{tabular}}
	\caption{$ \mathrm{STS}(15) $ $\# 80$}
	\label{tab:sts15-80}
\end{table}

\end{document}